\documentclass{amsart}
\usepackage{amssymb}
\newtheorem{theorem}{Theorem}[section]
\newtheorem{lemma}[theorem]{Lemma}
\newtheorem{maintheorem}[theorem]{Main Theorem}
\theoremstyle{remark}

\theoremstyle{corollary}
\newtheorem{corollary}[theorem]{Corollary}
\numberwithin{equation}{section}

\DeclareMathOperator{\Ext}{Ext}
\DeclareMathOperator{\Hom}{Hom}
\DeclareMathOperator{\Proj}{Proj}
\DeclareMathOperator{\Spec}{Spec}
\DeclareMathOperator{\codim}{codim}
\DeclareMathOperator{\Supp}{Supp}
\DeclareMathOperator{\injdim}{injdim}

\DeclareMathOperator{\hgt}{ht}
\DeclareMathOperator{\ch}{char}

\begin{document}

\title{Lyubeznik numbers for nonsingular projective varieties}
\author{Nicholas Switala}
\address{School of Mathematics\\ University of Minnesota\\ 127 Vincent Hall\\ 206 Church St. SE\\
Minneapolis, MN 55455}
\email{swit0020@math.umn.edu}
\thanks{NSF support through grant DMS-0701127 is gratefully acknowledged.}
\subjclass[2010]{Primary 13D45, 14F40}
\keywords{Local cohomology, algebraic de Rham cohomology, Lyubeznik numbers}

\begin{abstract}
In this paper, we determine completely the Lyubeznik numbers $\lambda_{i,j}(A)$ of the local ring $A$ at the vertex of the affine cone over a nonsingular projective variety $V$, where $V$ is defined over a field of characteristic zero, in terms of the dimensions of the algebraic de Rham cohomology spaces of $V$.  In particular, we prove that these numbers are intrinsic numerical invariants of $V$, even though \textit{a priori} their definition depends on an embedding into projective space.  This provides supporting evidence for a positive answer to the question of embedding-independence for arbitrary varieties in characteristic zero, which is still open.
\end{abstract}

\maketitle

\section{Introduction}

Let $A$ be a local ring that admits a surjection from an $m$-dimensional local ring $(R, \mathfrak{m})$ containing its residue field $k$, and let $I \subset R$ be the kernel of the surjection.  (All rings in this paper are commutative and have an identity.) For non-negative integers $i$ and $j$, the local cohomology multiplicities, or Lyubeznik numbers, of $A$ are defined to be $\lambda_{i,j}(A) = \dim_k(\Ext^i_R(k, H^{m-j}_I(R)))$, the $i^{th}$ Bass number of $H^{m-j}_I(R)$ with respect to $\mathfrak{m}$.  These numbers were defined in \cite{dmodules}.  It is proven in \cite[Theorem 3.4(d)]{dmodules} (resp. \cite[Theorem 2.1]{huneke}) that these Bass numbers are all finite when $\ch(k) = 0$ (resp. $\ch(k) = p > 0$), and further, they depend only on $A$, $i$, and $j$, not on the choice of $R$ nor on the choice of surjection $R \rightarrow A$ \cite[Theorem-Definition 4.1]{dmodules}.  Some properties of these numbers were subsequently worked out in \cite{walther} and \cite{kawasaki}, and a topological interpretation (in the case $k = \mathbb{C}$) was given in \cite{garcia}.\\

In this paper, we are concerned with the case in which $A$ is the local ring at the vertex of the affine cone over a projective variety.  To be precise, let $V$ be a projective variety of dimension $r$ over a field $k$ of characteristic zero.  Under an embedding $V \hookrightarrow \mathbb{P}^n_k$, we can write $V = \Proj(k[x_0, \ldots, x_n]/I)$ where $I$ is a homogeneous defining ideal for $V$.  Let $\mathfrak{m} = (x_0, \ldots, x_n)$ be the homogeneous maximal ideal of $k[x_0, \ldots, x_n]$, so that $I \subset \mathfrak{m}$.  Then $A = (k[x_0, \ldots, x_n]/I)_{\mathfrak{m}}$ is the local ring at the vertex of the affine cone over $V$, and we can define the multiplicities $\lambda_{i,j}(A)$ as in the preceding paragraph.  In \cite[p. 133]{survey}, Lyubeznik asked whether $\lambda_{i,j}(A)$ depend only on $V$, $i$, and $j$, and not on the embedding $V \hookrightarrow \mathbb{P}^n_k$ (or, for that matter, on $n$).  Zhang settled this question in the affirmative in the case of the ``top'' Lyubeznik number $\lambda_{r+1, r+1}(A)$, in any characteristic, in \cite{zhang}; he went on to give an affirmative answer for all $\lambda_{i,j}(A)$ in the $\ch(k) = p > 0$ case in \cite{zhang2}.  In \cite{zhang2}, several preliminary results are established in a characteristic-free setting, but the main line of argument makes crucial use of the Frobenius morphism.  This left the characteristic-zero case open for all but the ``top'' Lyubeznik number.\\

In this paper, we determine completely the Lyubeznik numbers $\lambda_{i,j}(A)$, in the case in which $V$ is a \textit{nonsingular} variety, in terms of quantities known already to be embedding-independent.  Our proof uses results of Hartshorne and Ogus on algebraic de Rham cohomology, and develops the ideas of \cite[p. 54]{dmodules} and \cite[Remark 2]{garcia}.  The characteristic-zero analogue of the result of \cite{zhang2} remains open for singular varieties.  Our result on the embedding-independence of the integers $\lambda_{i,j}(A)$ for nonsingular varieties provides further supporting evidence for a positive answer to the question of their embedding-independence for arbitrary varieties.\\

We fix some further notation.  $V$, $I$, and $A$ remain as above, except that now and for the remainder of the paper we will assume $V$ is nonsingular.  Write $R = (k[x_0, \ldots, x_n])_{\mathfrak{m}}$, a regular local ring, so that $A = R/I$; it is this $R$ which will intervene in the definition of $\lambda_{i,j}(A)$ in terms of Bass numbers.  The following quantities will appear repeatedly: $\dim(R) = \hgt(\mathfrak{m}) = n+1$, $\hgt(I) = \codim(V, \mathbb{P}^n_k) = n-r$, and $\dim(A) = r+1$. In particular, since $\dim(R) = n+1$, our definition of the Lyubeznik numbers of $A$ reads $\lambda_{i,j}(A) = \dim_k(\Ext^i_R(k, H^{n+1-j}_I(R)))$.\\

Finally, we remark that, as explained in \cite[$\S 8$]{zhang2}, it is harmless to assume that $k$ is algebraically closed, since Lyubeznik numbers are unaltered under extension of the base field. Under this assumption, we have the following embedding-independent description of $\lambda_{r+1, r+1}$:

\begin{theorem}\label{thm:1.1}
 \cite[Theorem 2.7]{zhang} Let $V_1, \ldots, V_s$ be the $r$-dimensional irreducible components of $V$, and let $\Gamma_V$ be the graph on the vertices $V_1, \ldots, V_s$ in which $V_i$ and $V_j$ are joined by an edge if and only if $\dim(V_i \cap V_j) = r-1$.  Then $\lambda_{r+1, r+1}(A)$ equals the number of connected components of $\Gamma_V$.
 \end{theorem}

Remarks in \cite[$\S 4$]{dmodules} establish that $\lambda_{i,j}(A) = 0$ if either $i$ or $j$ is greater than $r+1$, Theorem \ref{thm:1.1} deals with the case $i=j=r+1$, and our Main Theorem \ref{thm:1.2} below deals with the remaining cases.  This furnishes a complete description of all $\lambda_{i,j}(A)$.

\begin{maintheorem}\label{thm:1.2}

Write $\beta_j = \dim_k(H^j_{dR}(V))$, where $H_{dR}$ denotes algebraic de Rham cohomology in the sense of \cite{hartshorne}.   \cite[Theorem II.6.2]{hartshorne} implies that all $\beta_j$ are finite.  With the above notation and the hypothesis that $V$ is nonsingular, the following hold:

\begin{enumerate}
\item \label{claim:1} $\lambda_{i,j} = 0$ if $i>0$ and $j<r+1$;
\item \label{claim:2} $\lambda_{0,0} = 0, \lambda_{0,1} = \beta_0 - 1$;
\item \label{claim:3} $\lambda_{0,2} = \beta_1, \lambda_{0,j} = \beta_{j-1} - \beta_{j-3}$ for $j = 3, \ldots, r$;
\item \label{claim:4} $\lambda_{0, r+1} = \lambda_{1, r+1} = 0$ (cf. \cite[Theorem \textit{(c)}]{garcia});
\item \label{claim:5} $\lambda_{\ell, r+1} = \lambda_{0, r+2-\ell}$ (determined above) for $\ell = 2, \ldots, r$ (cf. \cite[Remark 1]{garcia}).
\end{enumerate}

\end{maintheorem}

Since algebraic de Rham cohomology is intrinsic to $V$ \cite[Theorem II.1.4]{hartshorne}, the above list, together with the result of \cite{zhang}, immediately implies the following:

\begin{corollary}\label{cor:1.3}
If $V$ is a nonsingular projective variety over a field $k$ of characteristic zero and $\lambda_{i,j}(A)$ is calculated as above, then for all $i$ and $j$, $\lambda_{i,j}(A)$ depends only on $V$, $i$, and $j$, not on $n$ nor on the embedding $V \hookrightarrow \mathbb{P}^n_k$.
\end{corollary}

The rest of the paper consists in a proof of the first three claims of Main Theorem \ref{thm:1.2}.  Claim \eqref{claim:1} will follow from known results on the local cohomology $H_I^{n+1-j}(R)$ and from the definition of $\lambda_{i,j}$.  To establish Claims \eqref{claim:2} and \eqref{claim:3}, we will first relate local cohomology supported at $I$ to local cohomology of the formal spectrum of the $I$-adic completion of $R$, then we will use a result of Ogus \cite{ogus} to relate this formal local cohomology to local de Rham cohomology.  Finally, we will appeal to an exact sequence worked out in \cite{hartshorne} connecting local de Rham cohomology at the vertex of the affine cone over $V$ to the de Rham cohomology of $V$ itself.  Claims \eqref{claim:4} and \eqref{claim:5} appeared first in \cite{garcia} and thus are not proven here; see also \cite{blickle} for an alternative proof using the Grothendieck composite-functor spectral sequence.\\

We would like to thank Professor Gennady Lyubeznik for suggesting the problem and for helpful conversations about the results in this paper.  We also thank Professor William Messing for profitable conversations about algebraic de Rham cohomology and the hard Lefschetz theorem.

\section{Proof of Claim \eqref{claim:1}}

We need the following lemma on the support of the local cohomology modules under consideration:

\begin{lemma}\label{lem:2.1}
With the above notation, since $V$ is nonsingular, $\Supp(H_I^i(R)) \subset \{\mathfrak{m}\}$ whenever $i \neq \hgt(I) = n-r.$
\end{lemma}

\begin{proof}
Since $V$ is nonsingular, the affine cone over $V$ has an isolated singularity at its vertex, and so all its other local rings are regular local rings.  That is, for any prime ideal $\mathfrak{p} \neq \mathfrak{m}$ of the coordinate ring $k[x_0, \ldots, x_n]/I$, the localization $(k[x_0, \ldots, x_n]/I)_{\mathfrak{p}} = (k[x_0, \ldots, x_n])_{\mathfrak{p}}/(I \cdot (k[x_0, \ldots, x_n])_{\mathfrak{p}}) \simeq R_{\mathfrak{p}}/I_{\mathfrak{p}} \simeq A_{\mathfrak{p}}$ is a regular local ring, where we write $I_{\mathfrak{p}} = I R_{\mathfrak{p}}$. We will show that for any prime ideal $\mathfrak{p} \neq \mathfrak{m}$ of $R$ (that is, any non-maximal homogeneous prime ideal of $k[x_0, \ldots, x_n]$), $\mathfrak{p}$ does not belong to the support of $H_I^i(R)$ for any $i \neq \hgt(I)$, \textit{i.e.} the localization $(H^i_I(R))_{\mathfrak{p}}$ is zero.  This is clear if $I \not \subset \mathfrak{p}$, so we assume that $I \subset \mathfrak{p}$, in which case $\hgt(I) = \hgt(I_{\mathfrak{p}})$. By the flat base change principle for local cohomology \cite[Theorem 4.3.2]{brodmann}, we have $(H^i_I(R))_{\mathfrak{p}} \simeq H^i_{I_{\mathfrak{p}}}(R_{\mathfrak{p}})$.  Since both $R_{\mathfrak{p}}$ and its quotient $R_{\mathfrak{p}}/I_{\mathfrak{p}} \simeq A_{\mathfrak{p}}$ are regular local rings, we conclude \cite[Proposition 2.2.4]{bruns} that $I_{\mathfrak{p}}$ is generated by part of a regular system of parameters of $R_{\mathfrak{p}}$, which must have $\hgt(I_{\mathfrak{p}}) = \hgt(I) = n-r$ elements.  But if an ideal is generated by a regular sequence of length $n-r$, local cohomology supported at this ideal cannot be nonzero in any degree other than $n-r$.
\end{proof}

\begin{corollary}\label{cor:2.2}
If $i > \hgt(I)$, $H_I^i(R)$ is an injective $R$-module.
\end{corollary}

\begin{proof}
By \cite[Theorem 3.4(b)]{dmodules}, $\injdim(H_I^i(R)) \leq \dim \Supp(H_I^i(R))$, and the right-hand side is zero if $\Supp(H_I^i(R)) \subset \{\mathfrak{m}\}$.
\end{proof}

Now suppose $j < r+1$.  Then $n+1-j > n-r = \hgt(I)$, so by the preceding corollary, $H^{n+1-j}_I(R)$ is an injective $R$-module.  This implies that $\Ext^i_R(k, H^{n+1-j}_I(R))$ vanishes for all $i>0$, so that if $i>0$ and $j < r+1$, the dimension of this Ext module (which, by definition, is $\lambda_{i,j}(A)$) is zero, proving Claim \eqref{claim:1}.\qed 

\section{Proofs of Claims \eqref{claim:2} and \eqref{claim:3}}

We are now concerned with computing $\lambda_{0,j}(A)$ for $0 \leq j \leq r$.  By definition, this is $\dim_k \Ext^0_R(k, H_I^{n+1-j}(R)) = \dim_k \Hom_R(k, H_I^{n+1-j}(R))$, the dimension of the socle of $H^{n+1-j}_I(R)$.  As discussed in the previous section, for $j < r+1$, $H^{n+1-j}_I(R)$ is an injective $R$-module supported only at $\mathfrak{m}$.  By the structure theory for injective modules over Noetherian rings, $H^{n+1-j}_I(R)$ is thus isomorphic to a direct sum of copies of $E = E(R/\mathfrak{m})$, the injective hull (as $R$-module) of the residue field $k = R/\mathfrak{m}$; by \cite[Theorem 3.4(d)]{dmodules}, the number of copies is finite.  Therefore we can write $H^{n+1-j}_I(R) \simeq E^{t_j}$ for some non-negative integer $t_j$.  Since the socle of $E$ is one-dimensional, the socle of $H^{n+1-j}_I(R) \simeq E^{t_j}$ is $t_j$-dimensional, from which it follows that $\lambda_{0,j}(A) = t_j$ (for $0 \leq j \leq r$).  We therefore turn our attention to the determination of $t_j$.\\

We may complete $R$ at the maximal ideal $\mathfrak{m}$ without changing the $\lambda_{i,j}(A)$ \cite[Lemma 4.2]{dmodules}, so we can, and will, assume that $R$ is the complete regular local ring $k[[x_0, \ldots, x_n]]$.  Working over a complete local ring, we have access to the theory of Matlis duality.  Let $D$ denote the Matlis dual functor from the category of $R$-modules to itself, so that $D(M) = \Hom_R(M, E)$ for an $R$-module $M$, where $E = E(R/\mathfrak{m})$ is the injective hull mentioned above.  We will compute $D(H^{n+1-j}_I(R))$ in two different ways and equate the two answers.  On the one hand, $D$ is an exact functor (since $E$ is injective) and $D(E) \simeq R$ \cite[Theorem 10.2.12(i)]{brodmann}, so that 
\[
D(H^{n+1-j}_I(R)) \simeq D(E^{t_j}) \simeq (D(E))^{t_j} \simeq R^{t_j}
\]
for each $j$ with $0 \leq j \leq r$.  On the other hand, by \cite[Proposition 2.2.3]{ogus}, we have isomorphisms $
D(H_I^{n+1-j}(R)) \simeq H_P^{j}(\hat{X}, \mathcal{O}_{\hat{X}})$ for all $j$.  Here $X = \Spec(R)$, $Y \subset X$ is the closed subscheme defined by $I$, $P \in Y$ is the closed point, $\hat{X}$ is the formal completion of $X$ along $Y$ and $\mathcal{O}_{\hat{X}}$ is the structure sheaf of $\hat{X}$.\\

Equating the results of our two calculations of $D(H_I^{n+1-j}(R))$, we see that $H_P^{j}(\hat{X}, \mathcal{O}_{\hat{X}}) \simeq R^{t_j}$.  So we have reduced ourselves to the calculation of the $R$-rank of $H_P^{j}(\hat{X}, \mathcal{O}_{\hat{X}})$, for which we will need algebraic de Rham cohomology.  We briefly recall the definition of algebraic de Rham cohomology, referring the reader to \cite{hartshorne} for details.  If $Y$ is a scheme of finite type over $k$ which admits an embedding as a closed subscheme of a smooth scheme $X$ over $k$, the algebraic de Rham cohomology $H^i_{dR}(Y)$ is by definition the hypercohomology $\mathbb{H}^i(\hat{X}, \hat{\Omega}_{X/k}^{\bullet})$ of the formal completion of the de Rham complex of $X$ along the closed subscheme $Y$. Replacing hypercohomology $\mathbb{H}^i$ with hypercohomology $\mathbb{H}^i_P$ supported in a closed point of $Y$, we get the definition of local algebraic de Rham cohomology $H^i_P(Y)$.  In \cite{hartshorne}, it is proven (Theorem II.1.4) that the definition of $H^i_{dR}(Y)$ is independent of the ambient smooth scheme $X$ and of the choice of embedding.  \cite[Theorem II.6.1]{hartshorne} establishes that the $H^i_{dR}(Y)$ are finite-dimensional $k$-vector spaces.  Analogous embedding-independence and finite-dimensionality statements for $H^i_P(Y)$ are proven in Chapter III of \cite{hartshorne}.\\

We will use the following theorem of Ogus to relate the formal local cohomology obtained above to algebraic de Rham cohomology (we have changed Ogus's notation to match ours):

\begin{theorem}\label{thm:3.1}
\cite[Theorem 2.3]{ogus} Let $k$ be a field of characteristic zero, let $R = k[[x_0, \ldots, x_n]]$, and let $Y$ be a closed subset of $X = \Spec(R)$, defined by an ideal $I$.  Let $\hat{X}$ be the formal completion of $X$ along $Y$, and let $P$ be the closed point.  Assume $s$ is an integer such that $\Supp(H_I^i(R)) \subset \{P\}$ for all $i > n+1-s$.  Then there are natural maps $R \otimes_k H^j_P(Y) \rightarrow H^j_P(\hat{X}, \mathcal{O}_{\hat{X}})$ which are isomorphisms for $j < s$ and injective for $j=s$, where $H^j_P(Y)$ denotes local algebraic de Rham cohomology in the sense defined above.
\end{theorem}

Notice that by Lemma \ref{lem:2.1}, the hypothesis of Ogus's theorem (that the local cohomology be supported at the maximal ideal) holds if we take $s = r + 1$.  The conclusion in our case is then that, for $j < r+1 = s$, we have an isomorphism $R \otimes_k H^j_P(Y) \simeq H^j_P(\hat{X}, \mathcal{O}_{\hat{X}})$, where the right-hand side is isomorphic to $R^{t_j}$.  Furthermore, by \cite[Theorem III.2.1]{hartshorne}, the local de Rham cohomology $H^j_P(Y)$ is a finite-dimensional $k$-vector space for each $j$.  This, together with the isomorphism $R \otimes_k H^j_P(Y) \simeq R^{t_j}$, implies that $\lambda_{0,j}(A) = t_j = \dim_k H^j_P(Y)$, which means we have reduced ourselves further to the calculation of the dimension of this de Rham cohomology space.  Ogus indicates a way to compute this dimension in \cite[Remark, p. 354]{ogus}.  For the convenience of the reader, we give the full details, showing how the explicit formulas of Claims \eqref{claim:2} and \eqref{claim:3} are obtained.\\

We remark that the calculation $\lambda_{0,j}(A) = \dim_k H^j_P(Y)$ is not new, and explain here how it follows from results in the literature.  By the ``Lefschetz principle'' it suffices to assume $k = \mathbb{C}$.  Blickle and Bondu prove in \cite[Theorem 1.2(1)]{blickle} that $\lambda_{0,j}(A) = \dim_{\mathbb{C}} H^j_{\{P\}}(Y; \mathbb{C})$, where this last denotes local singular cohomology with complex coefficients.  By the local de Rham-Betti comparison theorem \cite[\S IV.3, Remark]{hartshorne}, we have $H^j_{\{P\}}(Y; \mathbb{C}) \simeq H^j_P(Y)$.  This argument could have replaced the preceding five paragraphs.  However, the proof of Blickle and Bondu (which applies to cases more general than isolated singularities) requires substantial machinery of derived categories and intersection cohomology.  We wished to indicate a simpler algebraic proof in our case.\\

Hartshorne in \cite{hartshorne} relates the local de Rham cohomology spaces $H^j_P(Y)$ to global de Rham cohomology of the projective variety $V$ using exact sequences:

\begin{theorem}\label{thm:3.2}

\cite[Proposition III.3.2]{hartshorne} Let $V \subset \mathbb{P}^n_k$ be a projective variety, $C \subset \mathbb{A}^{n+1}_k$ the affine cone over $V$, and $P \in C$ the vertex.  Then $H_P^0(C) = 0$ and we have the following two exact sequences, where $H_P^i$ and $H_{dR}^i$ denote local algebraic de Rham cohomology and global algebraic de Rham cohomology, respectively:
\[
0 \rightarrow k \rightarrow H^0_{dR}(V) \rightarrow H_P^1(C) \rightarrow 0
\]
and
\[ 
0 \rightarrow H^1_{dR}(V) \rightarrow H^2_P(C) \rightarrow H^0_{dR}(V) \rightarrow H^2_{dR}(V) \rightarrow H^3_P(C) \rightarrow H^1_{dR}(V) \rightarrow \cdots
\]
Here the maps $H^i_{dR}(V) \rightarrow H^{i+2}_{dR}(V)$ for $i \geq 0$ are given by cup product with the class of a hyperplane section $\zeta \in H^2_{dR}(V)$.

\end{theorem}

By the ``strong excision theorem'' \cite[Proposition III.3.1]{hartshorne}, the $H^j_P(C)$ of the previous theorem (the de Rham cohomology of $C$ supported at the closed point $P$) and the $H^j_P(Y)$ appearing in Ogus's theorem (local de Rham cohomology of $\Spec(R/I) = \Spec(A)$) are the same, so that $\lambda_{0,j}(A) = \dim_k(H^j_P(C))$.\\

As in the statement of Main Theorem \ref{thm:1.2}, let $\beta_i$ denote the dimension of the $k$-vector space $H^i_{dR}(V)$.  Theorem \ref{thm:3.2}, which is true for any projective variety over a field of characteristic zero, is already enough for us to determine $\lambda_{0,0}(A) = 0$ (from the first statement) and $\lambda_{0,1}(A) = \beta_0 - 1$ (from the short exact sequence), proving Claim \eqref{claim:2}.  To prove Claim \eqref{claim:3} we will need the hypothesis that $V$ is nonsingular, and also the assumption (justified in the Introduction) that $k$ is algebraically closed.  These hypotheses allow us to extract more information from the long exact sequence of Theorem \ref{thm:3.2} by using the hard Lefschetz theorem in the following form:

\begin{theorem}\label{thm:3.3}
(Hard Lefschetz theorem for algebraic de Rham cohomology) If $V$ is a nonsingular projective variety of dimension $r$ over an algebraically closed field $k$ of characteristic zero, then for each $i$ with $0 \leq i \leq r$, the map $H^{r-i}_{dR}(V) \rightarrow H^{r+i}_{dR}(V)$ defined by $i$-fold cup product with the hyperplane section $\zeta$ is an isomorphism.
\end{theorem}

\begin{proof}
The classical hard Lefschetz theorem \cite[p. 122]{griffiths} is the analogue (with $k = \mathbb{C}$) of the preceding assertion with the de Rham cohomology of $V$ replaced by the singular cohomology $H^j(V^{an};\mathbb{C})$ of the associated complex-analytic space.  By Hartshorne's comparison theorem \cite[Theorem IV.1.1]{hartshorne}, we have, for all $j$, $H^j_{dR}(V) \simeq H^j(V^{an};\mathbb{C})$ (this only requires that $V$ be a scheme of finite type over $\mathbb{C}$), and this isomorphism carries the de Rham hyperplane class to the Betti hyperplane class, so that the result is true when the base field is $\mathbb{C}$.  The statement for an arbitrary algebraically closed base field $k$ with $\ch(k) = 0$ follows from the statement for $k = \mathbb{C}$ by the ``Lefschetz principle''.  Take a finite set of generators for the homogeneous defining ideal of $V$, in which only finitely many coefficients from $k$ appear, and consider the field $k'$ obtained by adjoining this finite set of coefficients to $\mathbb{Q}$.  $V$ may thus be defined over the field $k'$: if $V'$ is the variety defined over $k'$ by the same homogeneous polynomials as $V$, then $V = V' \times_{k'} k$, and $V'$ is also nonsingular.  Since $\mathbb{Q} \subset k' \subset k$ and $k$ is algebraically closed, we may embed $k'$ in $\mathbb{C}$.  As $V$ is nonsingular, $H^j_{dR}(V)$ is simply the hypercohomology of the de Rham complex of $V$, which is seen immediately to commute with extensions of the scalar field because formation of the module of K\"{a}hler differentials commutes with such extensions; this reduces the hard Lefschetz theorem for $V$ over $k'$ to the hard Lefschetz theorem for $V$ over $\mathbb{C}$, which we've already seen is true.
\end{proof}

We return to the proof of Claim \eqref{claim:3}.  If $j < r$ (so that $H^j_{dR}(V)$ is the source of one of the hard Lefschetz isomorphisms), the map $H^j_{dR}(V) \rightarrow H^{j+2}_{dR}(V)$ defined by cup product with $\zeta$, which occurs in the long exact sequence of Theorem \ref{thm:3.2}, is injective.  That long exact sequence hence splits into short exact sequences as follows:
\[
0 \rightarrow H^1_{dR}(V) \rightarrow H^2_P(C) \rightarrow 0
\]
and, for all $j \geq 3$,
\[
0 \rightarrow H^{j-3}_{dR}(V) \rightarrow H^{j-1}_{dR}(V) \rightarrow H^j_P(C) \rightarrow 0.
\]
We see at once from these exact sequences of finite-dimensional $k$-vector spaces that $\lambda_{0,2} = \dim_k H^2_P(C) = \beta_1$ and, for $j \geq 3$, $\lambda_{0,j} = \dim_k H^j_P(C) = \beta_{j-1} - \beta_{j-3}$, proving Claim \eqref{claim:3}.\qed


\begin{thebibliography}{[AAA99]}

\bibitem [BB05]{blickle} Blickle, M., R. Bondu, Local cohomology multiplicities in terms of \'etale cohomology, Ann. Inst. Fourier (Grenoble) \textbf{55} (7) (2005), 2239-2256.
\bibitem [BH98]{bruns} Bruns, W., J. Herzog, Cohen-Macaulay rings, revised edition, Cambridge Studies in Advanced Mathematics \textbf{39}, Cambridge University Press, Cambridge, 1988.
\bibitem [BS13]{brodmann} Brodmann, M., R. Sharp, Local cohomology, second edition, Cambridge Studies in Advanced Mathematics \textbf{136}, Cambridge University Press, Cambridge, 2013.
\bibitem [GH78]{griffiths} Griffiths, P., J. Harris, Principles of algebraic geometry, John Wiley and Sons, New York, 1978.
\bibitem [GLS98]{garcia} Garcia Lopez, R., C. Sabbah, Topological computation of local cohomology multiplicities, Collect. Math. \textbf{49} (2-3) (1998) 317-324.
\bibitem [Ha75]{hartshorne} Hartshorne, R., On the de Rham cohomology of algebraic varieties, Publ. Math. IHES \textbf{45} (1975), 5-99.
\bibitem [HS93]{huneke} Huneke, C., R. Sharp, Bass numbers of local cohomology modules, Trans. Amer. Math. Soc. \textbf{339} (1993), 765-779.
\bibitem [Kaw02]{kawasaki} Kawasaki, K.-I., On the highest Lyubeznik number, Math. Proc. Cambridge Philos. Soc. \textbf{132} (3) (2002) 409-417.
\bibitem [Lyu93]{dmodules} Lyubeznik, G., Finiteness properties of local cohomology modules, Invent. Math. \textbf{113} (1) (1993) 41-55.
\bibitem [Lyu02]{survey} Lyubeznik, G., A partial survey of local cohomology, pp. 121-154, in: Local cohomology and its applications (Guanajuato, 1999), Lecture Notes in Pure and Appl. Math. no. 226, Dekker, New York, 2002.
\bibitem [Lyu06]{gennady} Lyubeznik, G., On some local cohomology invariants of local rings, Math. Z. \textbf{254} (3) (2006) 627-640.
\bibitem [Og73]{ogus} Ogus, A., Local cohomological dimension of algebraic varieties, Ann. of Math. \textbf{98} (1973), 1-34.
\bibitem [Wal00]{walther} Walther, U., On the Lyubeznik numbers of a local ring, Proc. Amer. Math. Soc. \textbf{129} (6) (2001) 1631-1634.
\bibitem [Zha07]{zhang} Zhang, W., On the highest Lyubeznik number of a local ring, Compos. Math. \textbf{143} (1) (2007) 82-88.
\bibitem [Zha11]{zhang2} Zhang, W., Lyubeznik numbers of projective schemes, Adv. Math. \textbf{228} (2011), 575-616.

\end{thebibliography}
\end{document}